\documentclass[12pt]{amsart}
\usepackage[numbers, square]{natbib}
\usepackage{amssymb}
\usepackage{amsmath}
\usepackage{amsfonts}
\usepackage{geometry}
\usepackage{graphicx}
\usepackage{amsthm}
\usepackage{hyperref}
\usepackage[dvipsnames]{xcolor}

\providecommand{\U}[1]{\protect\rule{.1in}{.1in}}
\theoremstyle{plain}

\newtheorem{corollary}{Corollary}

\newtheorem{lemma}{Lemma}

\newtheorem{proposition}{Proposition}
\newtheorem{remark}{Remark}

\newtheorem{theorem}{Theorem}
\numberwithin{equation}{section}

\DeclareMathOperator{\dist}{dist} \DeclareMathOperator{\diam}{diam}

\begin{document}
\title[Upper bounds of nodal sets ]
{Upper bounds of nodal sets for eigenfunctions of eigenvalue problems}
\author {Fanghua Lin}
\address{
Department of Mathematics\\
Courant Institute of Mathematical Sciences \\
New York University\\
New York, NY 10012, USA\\
Email:   linf@cims.nyu.edu }
\author{ Jiuyi Zhu}
\address{
Department of Mathematics\\
Louisiana State University\\
Baton Rouge, LA 70803, USA\\
Email:  zhu@math.lsu.edu }
\thanks{Lin is supported in part by NSF grant DMS-1955249, Zhu is supported in part by NSF grant OIA-1832961}
\date{}

\begin{abstract}
 The aim of this article is to provide a simple and unified way to obtain the sharp upper bounds of nodal sets of eigenfunctions for different types of eigenvalue problems on real analytic domains. The examples include biharmonic Steklov eigenvalue problems, buckling eigenvalue problems and champed-plate eigenvalue problems. The geometric measure of nodal sets are derived from doubling inequalities and growth estimates
 for eigenfunctions. It is done through analytic estimates of Morrey-Nirenberg and Carleman estimates.
\end{abstract}

\maketitle
\section{Introduction}
The eigenvalue and
eigenfunction problems are {archetypical} in the theory of partial differential equations.
Different type of second order or higher order eigenvalue problems arise from physical phenomena in the literature. For instance, the famous Chaldni pattern is the nodal pattern modeled by the eigenfunctions of bi-Laplace eigenvalue problems. The Chladni pattern is the scientific, artistic, and even the sociological birthplace of the modern field of wave physics and quantum chaos. The goal of the paper is to provide a uniform way to obtain the upper bounds of nodal sets of eigenfunctions for various eigenvalue problems in real analytic domains. Since the nodal sets of eigenfunctions of Laplacian are well studied, we will focus on eigenfunctions of some higher order elliptic equations. The approach introduced in the paper also applies to the upper bounds of eigenfunctions of Laplacian with different boundary conditions in real analytic domains. Specifically, we consider three types of biharmonic Steklov eigenvalue problems
\begin{equation}
\left \{ \begin{array}{lll}
\triangle^2 e_\lambda=0  \quad &\mbox{in} \ {\Omega},  \medskip\\
e_\lambda=\triangle e_\lambda-\lambda \frac{\partial e_\lambda}{\partial \nu} =0 \quad &\mbox{on} \ {\partial\Omega},
\end{array}
\right.
\label{bistek1}
\end{equation}

\begin{equation}
\left \{ \begin{array}{lll}
\triangle^2 e_\lambda=0 \quad &\mbox{in} \ {\Omega},  \medskip\\
e_\lambda=\frac{\partial^2 e_\lambda}{\partial \nu^2} -\lambda \frac{\partial e_\lambda}{\partial \nu}=0    \quad &\mbox{on} \ {\partial\Omega}
\end{array}
\right.
\label{bistek2}
\end{equation}
and
\begin{equation}
\left \{ \begin{array}{lll}
\triangle^2 e_\lambda=0 \quad &\mbox{in} \ {\Omega},  \medskip\\
\frac{\partial e_\lambda}{\partial \nu}  =\frac{\partial \triangle e_\lambda}{\partial \nu}+\lambda^3e_\lambda=0    \quad &\mbox{on} \ {\partial\Omega},
\end{array}
\right.
\label{bistek3}
\end{equation}
where $\Omega\in \mathbb R^n$ with $n\geq 2$
is a bounded real analytic domain, $\nu$ is a unit outer normal, and $n$ is the dimension of the space in the paper. Those eigenvalue problems are important in biharmonic analysis, inverse problem and the theory of elasticity, see e.g. \cite{FGW}, \cite{KS}, \cite{P}.  If we consider the eigenfunctions in (\ref{bistek1})--(\ref{bistek3}) on the boundary, they become the eigenfunctions of  Neumann-to-Laplacian operator, Neumann-to-Neumann operator and Dirichlet to Neumann operator, respectively, see \cite{C}. The bi-Laplace equation arises in numerous problems of structural engineering. It models the displacements of a thin plate clamped near its boundary, the stresses
in an elastic body, the stream function in creeping flow of a viscous incompressible
fluid, etc. See, e.g. \cite{Me}.
Other typical bi-Laplace eigenvalue problems include
the buckling eigenvalue problem
\begin{equation}
\left \{ \begin{array}{lll}
\triangle^2 e_\lambda+\lambda \triangle e_\lambda=0  \quad &\mbox{in} \ {\Omega},  \medskip\\
e_\lambda=\frac{\partial e_\lambda}{\partial \nu}=0    \quad &\mbox{on} \ {\partial\Omega}
\end{array}
\right.
\label{bilap}
\end{equation}
and
the clamped-plate eigenvalue problem
\begin{equation}
\left \{ \begin{array}{lll}
\triangle^2 e_\lambda=\lambda e_\lambda  \quad &\mbox{in} \ {\Omega},  \medskip\\
e_\lambda=\frac{\partial e_\lambda}{\partial \nu}=0    \quad &\mbox{on} \ {\partial\Omega}.
\end{array}
\right.
\label{dirilap}
\end{equation}
 The buckling eigenvalue problem (\ref{bilap}) describes the critical buckling load of a clamped plate
subjected to a uniform compressive force around its boundary. The clamped-plate eigenvalue problem (\ref{dirilap}) arises from the vibration of a rigid thin plate with  clamped conditions.
For those eigenvalue problems,  there exists a sequence of eigenvalues $0\leq \lambda_1\leq \lambda_2<\cdots \to \infty$. Eigenfunctions $e_\lambda$ changes sign in $\Omega$ as $\lambda$ increases.

To find upper bounds of geometric measure of nodal sets of eigenfunctions for those eigenvalue problems in real analytic domains has been an interesting topic. For classical eigenfunctions on the smooth compact  Riemannian manifold
\begin{equation}
\triangle e_\lambda+\lambda e_\lambda=0 \quad \mbox{on} \ \mathcal{M},
\label{class}
\end{equation}
Yau \cite{Y} conjectured that the Hausdorff measure of nodal sets is bounded above and below as
\begin{align} c\sqrt{\lambda}\leq H^{n-1}(\{\mathcal{M}| e_\lambda(x)=0\})\leq C\sqrt{\lambda},
\label{conjecture}
\end{align}
where $c$, $C$ depend on the manifold $\mathcal{M}$. For the real analytic manifolds, the conjecture (\ref{conjecture}) was answered by Donnelly and Fefferman in their seminal paper \cite{DF}.
A relatively simpler proof using frequency functions for the upper
bound for general second order elliptic equations was given in \cite{Lin}. Let us review briefly
the recent literature concerning the progress of Yau's conjecture on nodal sets of classical eigenfunctions (\ref{class}). For the conjecture (\ref{conjecture}) on the measure of nodal sets on smooth manifolds, there are important breakthrough made by Logunov and Malinnikova \cite{LM},  \cite{Lo} and \cite{Lo1} in recent years. For the upper bounds of nodal sets on two dimensional manifolds,   Logunov and
Malinnikova \cite{LM} showed that $H^{1}(\{x\in\mathcal{M}| u(x)=0\})\leq C \lambda^{\frac{3}{4}-\epsilon} $, which slightly improve the upper bound $C \lambda^{\frac{3}{4}}$ by Donnelly and Fefferman \cite{DF2} and  Dong \cite{D}. For the upper bounds in higher dimensions $n\geq 3$ on smooth manifolds,  Logunov in \cite{Lo} obtained a polynomial upper bound
\begin{equation} H^{n-1}( \{x\in\mathcal{M} | u(x)=0\})\leq C \lambda^{\beta},
\label{poly}
\end{equation}
where $\beta>\frac{1}{2}$ depends only on the dimension. The polynomial upper bound (\ref{poly}) improves the exponential upper bound derived by  Hardt and Simon \cite{HS}.
For the lower bound, Logunov \cite{Lo1} completely solved the Yau's conjecture and obtained the sharp lower bound as
\begin{align} c\sqrt{\lambda}\leq H^{n-1}(\{x\in\mathcal{M}| u(x)=0\})
\label{lower} \end{align} for
 smooth manifolds for any dimensions. For $n=2$, such sharp lower bound was obtained earlier by Br\"uning \cite{Br}. This sharp lower bound (\ref{lower}) improves a polynomial lower bound obtained early by Colding and  Minicozzi \cite{CM}, Sogge and Zelditch \cite{SZ}. See also other polynomial lower bounds by different methods, e.g. \cite{HSo}, \cite{M}, \cite{S}.

Donnelly and Fefferman \cite{DF1} also considered the Dirichlet and Neumann eigenvalue problem on real analytic manifold $\mathcal{M}$ with boundary.
For the Dirichlet eigenvalue problem
\begin{equation}
\left \{ \begin{array}{rll}
-\triangle e_\lambda=\lambda e_\lambda \quad &\mbox{in} \ {\mathcal{M}},  \medskip\\
e_\lambda=0 \quad &\mbox{on} \ {\partial\mathcal{M}}
\end{array}
\right.
\label{Diri}
\end{equation}
and
Neumann eigenvalue problem
\begin{equation}
\left \{ \begin{array}{rll}
-\triangle e_\lambda=\lambda e_\lambda \quad &\mbox{in} \ {\mathcal{M}},  \medskip\\
\frac{\partial e_\lambda}{\partial \nu}=0 \quad &\mbox{on} \ {\partial\mathcal{M}},
\end{array}
\right.
\label{Neumann}
\end{equation}
the sharp lower bounds and upper bounds of the nodal sets as (\ref{conjecture}) were shown in \cite{DF1}.
Doubling inequalities are crucial in deriving the measure of nodal sets.
For the Dirichlet eigenvalue problem (\ref{Diri}) or Neumann eigenvalue problem (\ref{Neumann}) of the Laplacian, one is able to construct a doubling manifold by an odd or even extension of eigenfunctions to get rid of the boundary. Then one can derive the doubling inequalities in the double manifold using Carleman estimates, see \cite{DF1}.

In analogy to the biharmonic Steklov eigenvalue problems, the Steklov eigenvalue problem for Laplacian is given by
 \begin{equation}
\left \{ \begin{array}{rll}
\triangle e_\lambda=0 \quad &\mbox{in} \ {\Omega},  \medskip\\
\frac{\partial e_\lambda}{\partial \nu}=\lambda e_\lambda \quad &\mbox{on} \ {\partial\Omega}.
\end{array}
\right.
\label{Stek}
\end{equation}
The study of nodal sets for Steklov eigenfunctions was initiated in \cite{BL}.
The sharp upper bounds of interior nodal sets of eigenfunctions (\ref{Stek}) on real analytic surface was shown in \cite{PST}. The sharp upper bounds of interior nodal sets for Steklov eigenfunctions was generalized to any dimensions by Zhu in \cite{Zh5}. The sharp upper bounds of boundary nodal sets of eigenfunctions (\ref{Stek}) was obtained by Zelditch in \cite{Z1}. Interested readers may also refer to some other literature on the lower bounds or upper bounds of nodal sets of Steklov eigenfunctions, see e.g. \cite{WZ}, \cite{SWZ}, \cite{Zh1}, \cite{Zh4}, and other related topics in \cite{BG}, \cite{GT}. To obtain the upper bounds of nodal sets in \cite{BL} and \cite{Zh5}, an auxiliary function was introduced to reduce the Steklov eigenvalue problem (\ref{Stek}) into an elliptic equation with Neumann boundary condition. Then one is able to construct the double manifold by an even extension. The doubling inequalities are derived on the double manifold using Carleman estimates.

This aforementioned strategy does not seem to be applicable for those bi-Laplace operators,
general
eigenvalue problems (\ref{bistek1})--(\ref{dirilap}), or even Laplace operators
with Robin type boundary conditions involving boundary potential functions, since the double manifold is not
available. Another way to obtain the bounds of nodal sets is to use analytical continuation of the wave kernel, which is a global method applying the machinery of Fourier integral operators, see e.g. \cite{Z}, \cite{TZ}. This approach was adapted to prove the upper bound of boundary nodal sets of Steklov eigenfunctions in \cite{Z1}. The biharmoinc Steklov eigenfunctions in (\ref{bistek1})--(\ref{bistek3}) can be considered as eigenfunctions of elliptic pseudo-differential operators on the boundary.
It seems that the method can give the upper bounds of nodal sets of eigenfunction (\ref{bistek1})--(\ref{bistek3}) in a tubular  neighborhood of the boundary, but may need further analysis in the deep interior away from the tubular neighborhood. Since the doubling inequalities are local estimates, we aim to apply the doubling inequalities in the domain including the deep interior.
We adopt a approach to obtain the doubling inequalities in the domain, which is applicable for general eigenvalue problems. Our
strategy works as follows. Combining a lifting argument and analyticity results, we can
do a real analytic continuation for eigenfuctions in an extended domain and the extended
functions have some controlled growth. By this lifting argument, we hide the dependence of eigenvalues in the analytic continuation argument so that the extended domain is independent of the eigenvalues. Furthermore, it provides the growth control of the extended function, See Proposition 1 below for more details. Because of the extension, we do not need to distinguish between the tubular neighborhood and the deep interior. Relied on the growth control estimates and Carleman estimates, we are able to provides the doubling inequalities in the domain including the deep interior.  The measure of nodal sets follows from the doubling inequalities and the complex growth lemma.

For those biharmonic Steklov eigenvalue problems (\ref{bistek1})--(\ref{bistek3}), some polynomial lower bound estimates for nodal sets of eigenfunctions $e_\lambda$ in smooth manifolds in spirit of \cite{SZ}, \cite{WZ}, and \cite{SWZ} was obtained by Chang in \cite{C}. We can show the following results on the upper bounds of the measure of nodal sets on real analytic domains.
\begin{theorem}
Let $e_\lambda$ be the eigenfunction in (\ref{bistek1}), (\ref{bistek2}) or (\ref{bistek3}). There exists a positive constant $C$ depending only on the real analytic domain $\Omega$ such that
\begin{equation} H^{n-1}(\{x\in \Omega |e_\lambda(x)=0\})\leq C\lambda.
 \end{equation}
\label{th1}
\end{theorem}

The proof of Theorem \ref{th1} sets a model for our approach in obtaining the upper bounds of nodal sets in real analytic domains. Theorem \ref{th2} and \ref{th3} follow more or less the similar strategy. However, some different arguments are used to derive the doubling inequalities in these theorems.
For the  bi-Laplace buckling eigenvalue problem, we can show the following upper bounds.
\begin{theorem}
Let $e_\lambda$ be the eigenfunction in (\ref{bilap}). There exists a positive constant $C$ depending only on the real analytic domain $\Omega$ such that
\begin{equation} H^{n-1}(\{x\in \Omega |e_\lambda(x)=0\})\leq C\sqrt{\lambda}. \end{equation}
\label{th2}
\end{theorem}

For the clamped-plate eigenvalue problem, the following upper bounds can be derived.
\begin{theorem}
Let $e_\lambda$ be the eigenfunction in (\ref{dirilap}). There exists a positive constant $C$ depending only on the real analytic domain $\Omega$ such that
\begin{equation} H^{n-1}(\{x\in \Omega |e_\lambda(x)=0\})\leq C\lambda^{\frac{1}{4}}.
 \end{equation}
\label{th3}
\end{theorem}

Note that the different powers of $\lambda$ in Theorem 1--3 basically come from the rescaling argument. Hence, those are sharp upper bounds for the measure of nodal sets of eigenfunctions. For the nodal sets of higher order elliptic equations in real analytic domains, Kukavica \cite{Ku} showed another way to obtain the upper bounds of nodal sets of eigenfunctions based on a regularity result
by elliptic iterations
and an estimate on zero sets of real-analytic functions due to Donnelly-Fefferman \cite{DF}. It seems that such approach can not work for the biharmonic Steklov eigenvalue problem (\ref{bistek1})--(\ref{bistek3}) and bi-Laplace buckling eigenvalue problem (\ref{bilap}).

The organization of the article is as follows. Section
 2 is devoted to the upper bounds of nodal sets for biharmonic Steklov eigenfunctions (\ref{bistek1})--(\ref{bistek3}).  We first derive the real analytic continuation for eigenfunctions, then show the doubling inequalities. The vanishing order of eigenfunctions is obtained as a consequence of the doubling inequalities.
  In section 3, we prove the upper bounds of nodal sets for eigenfunctions of buckling problems (\ref{bilap}). Section 4 is used to show the upper bounds for  nodal sets for champed-plate problems. The upper bounds of nodal sets for eigenfunctions of higher order elliptic equations of arbitrary order with Dirichlet and Navier boundary conditions are also shown.
The letters $C$, $C_i$, $C_i(n, \partial\Omega)$ denote generic positive
constants that do not depend on $e_\lambda$ or $\lambda$, and may vary from line to
line. In the paper, since we study the asymptotic properties of eigenfunctions, we assume that the eigenvalue $\lambda$ is large. The approach of the paper for the nodal sets of eigenfunctions can be applied to real analytic  Riemannian  manifolds with boundary.

\section{Nodal sets of biharmonic Steklov eigenfunctions}
This section is devoted to obtaining the upper bounds of nodal sets  of biharmonic Steklov eigenfunctions. We first analytically extend $e_\lambda$ into a bigger domain that includes $\Omega$.
 We apply lifting arguments and the analyticity to do the real analytic continuation.
\begin{proposition}
Let $e_\lambda$ be the eigenfunction in (\ref{bistek1}), (\ref{bistek2}) or (\ref{bistek3}) in the real analytic bounded domain $\Omega$. Then $e_\lambda$ can be analytically extended to a bounded domain $\widetilde{\Omega} \supset\Omega$ and
\begin{align}
\|e_\lambda\|_{L^\infty(\widetilde{\Omega})}\leq e^{C\lambda}\|e_\lambda\|_{L^\infty(\Omega)}
\label{want}
\end{align}
for some $C$ depending only on $\partial\Omega$.
\label{pro1}
\end{proposition}
 \begin{proof} Let us first consider the eigenvalue problem (\ref{bistek1}) as an example. Since $\Omega$ is a real analytic domain, by standard regularity theorems for elliptic equations, $e_\lambda(x)$ is real analytic in $\bar \Omega$, see e.g. \cite{MN} or section 6.6 in \cite{Mo}.
 We hope to extend $e_\lambda$ across the boundary $\partial\Omega$ analytically. Instead of examining the dependence of $\lambda$ in the extension, we provide an elementary way to get rid of the eigenvalue $\lambda$ on the boundary and perform the extension for a new equation without $\lambda$. We adopt the following lifting arguments. Let $$\hat{u}(x, t)=e^{\lambda t} e_\lambda(x).$$ Then the new function $\hat{u}(x, t)$ satisfies the following equation
 \begin{equation}
\left \{ \begin{array}{lll}
\triangle^2 \hat{u}+\partial^4_t \hat{u}-\lambda^4 \hat{u}  =0  \quad &\mbox{in} \ {\Omega}\times (-\infty, \infty),  \medskip\\
\hat{u}=\triangle \hat{u}- \frac{\partial^2 \hat{u}}{\partial t \partial \nu} =0 \quad &\mbox{on} \ {\partial\Omega}\times (-\infty, \infty)
\end{array}
\right.
\end{equation}
for any $t\in (-\infty, \infty)$.
 To remove the eigenvalue $\lambda$ in the equation, we perform another lifting argument. Let $${u}(x, t, s)=e^{\sqrt{i} \lambda s} \hat{u}(x, t)$$
for any $s\in  (-\infty, \infty)$,  where $i$ is the imaginary unit. Then ${u}(x, t, s)$ satisfies the equation
 \begin{equation}
\left \{ \begin{array}{lll}
\triangle^2 {u}+\partial^4_t {u}+\partial^4_s {u}  =0  \quad &\mbox{in} \ {\Omega}\times (-\infty, \infty)\times (-\infty, \infty),  \medskip\\
{u}=\triangle {u}- \frac{\partial^2 {u}}{\partial t \partial \nu} =0 \quad &\mbox{on} \ {\partial\Omega}\times (-\infty, \infty)\times (-\infty, \infty).
\end{array}
\right.
\label{newlift}
\end{equation}
 Note that the equation (\ref{newlift}) is uniformly elliptic. We apply Fermi coordinates near the boundary to flatten the boundary $\partial\Omega$.   We can find a small constant $\rho>0$ so that there exists a map $(x', \ x_n)\in \partial \Omega\times [0, \ \rho)\to \Omega$ sending $(x', \ x_n)$ to the endpoint $x\in \Omega$ with length $x_n$, which starts at $x'\in \partial \Omega$ and is perpendicular to $\partial\Omega$. Such map is a local diffeomorphism. Notice that $x'$ is the geodesic normal coordinates of $\partial\Omega$ and $x_n=0$ is identified locally as $\partial\Omega$. The metric takes the form
$$ \sum^n_{i,j=1} g_{ij} dx^i dx^j= dx_n^2+ \sum^{n-1}_{i,j=1} g'_{ij}(x', x_n) dx^i dx^j,   $$
where  $g'_{ij}(x', x_n)$ is a Riemannian metric on $\partial \Omega$ depending analytically on $x_n\in [0, \ \rho).$ In a neighborhood of the boundary, the Laplacian can be written as
\begin{align} \triangle =\sum^n_{i,j=1} g^{ij}\frac{\partial^2}{\partial x_i \partial x_j}+\sum^n_{i=1} q_i(x)\frac{\partial}{\partial x_i}
\label{newlap}
\end{align}
using local coordinates for $\partial \Omega$, where  $g^{ij}$ is the matrix with entries $(g^{ij})_{1<i\leq j<n-1}= (g'_{ij})^{-1}$ and $g^{nn}=1$ and $g^{nk}=g^{kn}=0$ for $k\not =n$. Moreover, $g^{ij}$ and $ q_i(x)$ are real analytic functions because $\partial\Omega$ is real analytic.
 For any $x_0\in \partial \Omega$, by rotation and translation, we may assume $x_0$ as the origin.
Introduce the ball as
\begin{align*}
\Omega_{R}=\{ (x, t, s)\in \mathbb R^{n+2}| |x|<R,  \ |t|<R, \ |s|<R\}
\end{align*}
and the half-ball as
\begin{align*}
\Omega^+_{R}=\{ (x, t, s)\in \mathbb R^{n+2} | |x|<R \ \mbox{with} \   x_n\geq 0, \ |t|<R, \ |s|<R\}.
\end{align*}

By rescaling, we may consider the function $u(x, t, s )$  locally in the half-ball with the flatten boundary by Fermi coordinates.  Thus, $u(x, t, s )$ satisfies
\begin{equation}
\left \{ \begin{array}{lll}
\triangle^2 u+\partial_t^4 u+\partial_s^4 u =0  \quad &\mbox{in} \ \Omega^+_{2},  \medskip\\
u=\triangle u- \frac{\partial^2 u}{\partial t\partial x_n} =0 \quad &\mbox{on} \  \Omega^+_{2}\cap\{x_n=0\}.
\end{array}
\right.
\label{backgo}
\end{equation}

We can check as in \cite{ADN} that (\ref{backgo}) is a uniformly elliptic equation with boundary conditions satisfying the complementing conditions. Notice also that the equation (\ref{backgo}) is independent of $\lambda$. By the analyticity results in \cite{MN}, \cite{Mo} (section 6.6), the solution $u(x, t, s)$ is analytic on $\Omega^+_{2}\cap\{x_n=0\}$ with radius of convergence exceeding some constant $\delta$ depending only on $\Omega$ and $n$. Thus, $u(x, t, s)$ can be analytically extended to $\Omega_{\delta}$. Moreover, we have
\begin{align}
\|u\|_{L^\infty(\Omega_{\delta})}\leq  C(n, \Omega)\|u\|_{L^\infty (\Omega^+_{2})}.
\label{ubb}
\end{align}
Since the boundary $\partial\Omega$ is compact and the equation is invariant under the translation with respect to the variable $t$ and $s$, applying those arguments in a finite number of neighborhoods that cover $\partial\Omega\times [-1, \ 1]\times [-1, \ 1]$, we can extend the eigenfunction $u(x, t, s)$ to a neighborhood $\widehat{\Omega}_1=\{(x, t, s)\in \mathbb R^{n+2}|
\dist(x, \Omega)\leq {\hat{C}(n, \partial\Omega)}, |t|\leq 1, \ |s|\leq 1\}$. Let $\widehat{\Omega}=\{(x, t, s)\in \mathbb R^{n+2}|
x\in\Omega, |t|\leq 2, \ |s|\leq 2\}$. It follows from (\ref{ubb}) that
\begin{align}
\|u\|_{L^\infty(\widehat{\Omega}_1)}\leq C(n, \partial\Omega) \|u\|_{L^\infty(\widehat{\Omega})}.
\label{firstexx}
\end{align}
By the uniqueness of the analytic continuation, it  follows that
\begin{align}
\triangle^2 u+\partial^4_t u+\partial^4_s u=0 \quad \mbox{in} \ \widehat{\Omega}_1.
\end{align}
From the definition $u(x,t,s)=e^{\lambda t}{ e^{\sqrt{i} \lambda s}} e_\lambda(x)$ and the uniqueness of the analytic continuation again, we have that
\begin{align}
\triangle^2 e_\lambda=0 \quad \mbox{in} \ {\widetilde{\Omega}},
\label{last}
\end{align}
where $\widetilde{\Omega}=\{ x\in \mathbb R^n|\dist(x, \Omega)\leq d\}$ for $d\leq {\hat{C}(n, \partial\Omega)}$. Furthermore,
 it is readily from (\ref{firstexx}) and definition of $u$ that
\begin{align}
\|e_\lambda\|_{L^\infty(\widetilde{{\Omega}})}\leq e^{C\lambda} \|e_\lambda\|_{L^\infty({\Omega})},
\label{firstex}
\end{align}
Therefore, the conclusion (\ref{want}) is achieved for eigenfunctions in (\ref{bistek1}).

For eigenvalue problems (\ref{bistek2}), we adopt the same approach. Let
\begin{align}
u(x, t, s)=e^{\lambda t} e^{\sqrt{i}\lambda s} e_\lambda(x).
\end{align}
Then $u(x, t, s)$ satisfies the equation
 \begin{equation}
\left \{ \begin{array}{lll}
\triangle^2 {u}+\partial^4_t {u}+\partial^4_s {u}  =0  \quad &\mbox{in} \ {\Omega}\times (-\infty, \infty)\times (-\infty, \infty),  \medskip\\
{u}=\frac{\partial^2 u}{\partial \nu^2}- \frac{\partial^2 {u}}{\partial t \partial \nu} =0 \quad &\mbox{on} \ {\partial\Omega}\times (-\infty, \infty)\times (-\infty, \infty).
\end{array}
\right.
\label{newlift2}
\end{equation}
The equation (\ref{newlift2}) is also a uniformly elliptic with boundary conditions satisfying the complementing conditions.
Following the procedure as performed for the eigenvalue problem (\ref{bistek1}), we can also analytically extend $u(x, t, s)$ across the boundary $\partial\Omega\times [-1, \ 1]\times [-1, \ 1]$ and obtain that
\begin{align}
\triangle^2 e_\lambda=0 \quad \mbox{in} \ {\widetilde{\Omega}}
\end{align}
which satisfies the controlled growth
\begin{align}
\|e_\lambda\|_{L^\infty({\widetilde{\Omega}})}\leq e^{C\lambda} \|e_\lambda\|_{L^\infty({\Omega})}.
\end{align}

For the eigenvalue problem (\ref{bistek3}), the same arguments apply as well. Again we choose
\begin{align}
u(x, t, s)=e^{\lambda t} e^{\sqrt{i}\lambda s} e_\lambda(x).
\end{align}
 Then  $u(x, t, s)$ satisfies the equation
 \begin{equation}
\left \{ \begin{array}{lll}
\triangle^2 {u}+\partial^4_t {u}+\partial^4_s {u}  =0  \quad &\mbox{in} \ {\Omega}\times (-\infty, \infty)\times (-\infty, \infty),  \medskip\\
\frac{\partial u}{\partial \nu}=\frac{\partial \triangle u}{\partial \nu}+\frac{\partial^3 {u}}{\partial t^3 } =0 \quad &\mbox{on} \ {\partial\Omega}\times (-\infty, \infty)\times (-\infty, \infty).
\end{array}
\right.
\label{newlift3}
 \end{equation}
 Thus, the estimates (\ref{firstex}) and (\ref{last}) holds for the eigenfunctions in (\ref{bistek3}).
This completes the proof of the Proposition.
\end{proof}

\begin{remark}
Such real analytic continuation result also holds for a range of eigenvalue problems with boundary. Obviously, it works for  Dirichlet eigenvalue problems (\ref{Diri}),  Neumann eigenvalue problems (\ref{Neumann}) for Laplacian, and Steklov eigenvalue problems (\ref{Stek}). The power $C\lambda$ of $e^{C\lambda}$ in (\ref{want})  is from the rescaling argument. The key ingredients of the proof are the lifting arguments and the analyticity results.
\end{remark}
To derive bounds of nodal sets for eigenfunctions, a crucial step is to obtain the doubling inequality estimates. Such estimates control the growth of eigenfunctions locally. To obtain the doubling inequalities,  three-ball inequalities are used.
Next we establish the three-ball inequality for $e_\lambda$. Then
we will establish doubling inequalities for balls centered at any point in $\Omega$. Carleman estimates are efficient tools to  obtain those three-ball inequalities and doubling inequalities. Another popular tool for those inequalities is the frequency function, see e.g. \cite{GL}. Let us introduce some notations.
  If not specified, $\|\cdot\|$ or $\|\cdot\|_R$ is denoted as the $L^2$ norm centered at the ball $\mathbb B_R$. Let $\phi(x)=-\ln r(x)+r^\epsilon(x)$ be the weight function, where $r(x)=|x-x_0|$ be the distance to some point $x_0\in\Omega$ and $0<\epsilon<1$ is some small number.  Such weight function $\phi(x)$ was introduced by H\"ormander in \cite{H1}. The following quantitative Carleman estimates were established in \cite{Zh3} for bi-Laplace operators.
\begin{lemma}
There exist positive constants $R_0$ and $C$,  such that, for any $x_0\in \Omega$, any smooth function $f\in C^\infty_0(\mathbb B_{R_0}(x_0)\backslash \mathbb B_{\delta}(x_0))$ with $0<\delta<R_0<1$  and $\tau>C$, one has
\begin{align}
C\|r^4 e^{\tau \phi} \triangle^2 f\|\geq \tau^3 \|r^\epsilon e^{\tau \phi} f\|+\tau^2 \delta^2 \|r^{-2} e^{\tau \phi} f\|.
\label{Carle}
\end{align}
 \label{lemmm}
\end{lemma}

Thanks to the Carleman estimates (\ref{Carle}), for $e_\lambda(x)$ in (\ref{last}),  it is standard to  establish the three-ball inequality
\begin{align}
\|e_\lambda\|_{L^2(\mathbb B_{2R}(x_0))}\leq C \|e_\lambda\|^\beta_{L^2(\mathbb B_{R}(x_0))}\|e_\lambda\|^{1-\beta}_{L^2(\mathbb B_{3R}(x_0))}
\label{threee}
\end{align}
for $0<R<R_0$, $x_0\in \Omega$ and $0<\beta<1$. We may choose $R_0<\frac{d}{10}$. Recall that $d=\dist(\Omega, \ \partial\widetilde{\Omega})$. Standard elliptic estimates imply the $L^\infty$ norm three-ball inequality. We still write it as
\begin{align}
\|e_\lambda\|_{L^\infty(\mathbb B_{2R}(x_0))}\leq C \|e_\lambda\|^\beta_{L^\infty(\mathbb B_{R}(x_0))}\|e_\lambda\|^{1-\beta}_{L^\infty(\mathbb B_{3R}(x_0))}.
\label{three}
\end{align} Similar arguments have also been carried out in Lemma \ref{baa}. Interested readers may refer to Lemma \ref{baa} in the paper or \cite{Zh3} for details of the argument.

For any $\hat{x}\in \Omega$, we will derive the estimate
\begin{align}
\|e_\lambda\|_{L^\infty(\mathbb B_{R}(\hat{x}))}\geq e^{-C(R){\lambda}}\|e_\lambda\|_{L^\infty(\tilde{\Omega})},
\label{claim}
\end{align}
where $C(R)$ is a positive constant depending on $R$. The estimate (\ref{claim}) is a quantitative result for the norm of $e_\lambda$ in any ball in $\tilde{\Omega}$ centered at some point in $\Omega$. We shall show (\ref{claim}) by iteration of the three-ball inequality.
Let $|e_\lambda(\bar x)|=\sup_{x\in \Omega}|e_\lambda(x)|$. We do a propagation of smallness using the three-ball inequality (\ref{three}) to get to $\bar x$ from $\hat{x}$.
Applying the three ball inequality (\ref{three}) at $\hat{x}$ and (\ref{want}) in Proposition 1, we have
\begin{align}
\|e_\lambda\|_{L^\infty(\mathbb B_{2R}(\hat{x}))}\leq e^{C(1-\beta){\lambda}} \|e_\lambda\|^\beta_{L^\infty(\mathbb B_{R}(\hat{x}))}\|e_\lambda\|^{1-\beta}_{L^\infty(\Omega)}.
\label{norma}
\end{align}
Without loss of generality, let us normalize $\|e_\lambda\|_{L^\infty(\Omega)}=1$. Then
\begin{align}
\|e_\lambda\|_{L^\infty(\mathbb B_{2R}(\hat{x}))}\leq e^{C{\lambda}} \|e_\lambda\|^\beta_{L^\infty(\mathbb B_{R}(\hat{x}))}.
\end{align}
Choose $x_1\in \mathbb B_{R}(\hat{x})$ such that $\mathbb B_R(x_1)\subset \mathbb B_{2R}(\hat{x})$, it follows that
\begin{align}
\|e_\lambda\|_{L^\infty(\mathbb B_{R}(x_1))}\leq e^{C{\lambda}} \|e_\lambda\|^\beta_{L^\infty(\mathbb B_{R}(\hat{x}))}.
\end{align}
The application of the three-ball inequality (\ref{three}) at $x_1$ yields that
\begin{align}
\|e_\lambda\|_{L^\infty(\mathbb B_{2R}(x_1))}\leq e^{C{\lambda}} \|e_\lambda\|^{\beta^2}_{L^\infty(\mathbb B_{R}(\hat{x}))}.
\end{align}
Fix such $R$, we choose a sequence of balls $\mathbb B_{R}(x_i)$ centered at $x_i$ such that $x_{i+1}\in \mathbb B_R(x_i)$ and $\mathbb B_R(x_{i+1})\subset \mathbb B_{2R}(x_i)$. After finitely many of steps, we could get to the point $\bar x$ where $e_\lambda(\bar x)=1$, that is, $\hat{x}, x_1, \cdots, x_m=\bar x$. The number of $m$ depends on $R$ and $diam(\Omega)$. Repeating the three-ball inequality (\ref{three}) at those $x_i$, $i=2, 3, \cdots, m$, we arrive at
\begin{align}
\|e_\lambda\|_{L^\infty(\mathbb B_{R}(x_m))}\leq e^{C{\lambda}} \|e_\lambda\|^{\beta^m}_{L^\infty(\mathbb B_{R}(\hat{x}))}.
\end{align}
Since $0<\beta<1$, we obtain that
\begin{align}
\|e_\lambda\|_{L^\infty(\mathbb B_{R}(\hat{x}))}&\geq e^{-\frac{C{\lambda}}{\beta^m}} \|e_\lambda\|_{L^\infty(\Omega)}\nonumber \\
&\geq e^{-{C(R){\lambda}}} \|e_\lambda\|_{L^\infty(\Omega)}.
\label{quant}
\end{align}
Thus, the estimate (\ref{claim}) is verified because of (\ref{want}). By rescaling, it also holds that
\begin{align}
\|e_\lambda\|_{L^\infty(\mathbb B_{\frac{R}{12}}(\hat{x}))}\geq e^{-C(R){\lambda}}\|e_\lambda\|_{L^\infty(\tilde{\Omega})}
\label{claimm}
\end{align}
for any $\hat{x}\in \Omega$ and $0<R<R_0$.

Define the annulus $ A_{R_1, \ R_2}({x_0}):=\{x\in \mathbb R^n| R_1\leq |x-x_0|\leq R_2\}$.
For any $x_0\in \Omega$, there exist some point $\hat{x}$ such that $\mathbb B_{\frac{R}{12}}(\hat{x}) \subset A_{\frac{R}{2}, \ \frac{2R}{3}}(x_0)$.
Therefore, (\ref{claimm}) also implies that
\begin{align}
\|e_\lambda\|_{L^\infty( A_{\frac{R}{2}, \ \frac{3R}{2}}({x_0}))}
&\geq e^{-{C(R){\lambda}}} \|e_\lambda\|_{L^\infty(\tilde{\Omega})}.
\label{exdouex}
\end{align}

Now we derive the quantitative doubling inequalities from Carleman estimates (\ref{Carle}), the estimates (\ref{claim}) and (\ref{exdouex}). See e.g. \cite{AMRV} for some qualitative doubling inequalities for solutions of elliptic systems.
\begin{proposition}
Let $e_\lambda$ be the eigenfunction in (\ref{bistek1}), (\ref{bistek2}) or (\ref{bistek3}). There exists a positive constant $C$ depending only on the real analytic domain $\Omega$ such that
\begin{equation}
\|e_\lambda\|_{L^\infty(\mathbb B_{2r}(x_0))}\leq e^{C\lambda}\|e_\lambda\|_{L^\infty(\mathbb B_{r}(x_0))}
\label{LLL}
\end{equation}
for any $x_0\in \overline{\Omega}$ and $0<r\leq \frac{d}{4}$.
\label{pro2}
\end{proposition}

\begin{proof}
Let us fix $R=\frac{ R_0}{8}$, where $ R_0$ is the one in the three-ball inequality (\ref{three}). Let $0<\delta<\frac{R}{24}$ be arbitrarily small. Let $r(x)=|x-x_0|$. We introduce a smooth cut-off function $0<\psi<1$  as follows,
\begin{itemize}
\item $\psi(r)=0$ \ \ \mbox{if} \ $r(x)<\delta$ \ \mbox{or} \  $r(x)>2R$, \medskip
\item $\psi(r)=1$ \ \ \mbox{if} \ $\frac{3\delta}{2}<r(x)<R$, \medskip
\item $|\nabla^\alpha \psi|\leq \frac{C}{\delta^\alpha}$ \ \ \mbox{if} $\delta<r(x)<\frac{3\delta}{2}$,\medskip
\item $|\nabla^\alpha \psi|\leq C$ \ \  \mbox{if} \ $R<r(x)<2R$.
\end{itemize}
We apply the Carleman estimates (\ref{Carle}) to obtain the doubling inequalities. Replacing $f$ by $\psi e_\lambda$ and substituting it into (\ref{Carle}) yields that
\begin{align*}
 \tau^3\| r^{\epsilon} e^{\tau \phi}\psi  e_\lambda\|+ \tau^2\delta^2 \| r^{-2} e^{\tau \phi}\psi  e_\lambda\|
\leq C\| r^4 e^{\tau \phi}[ \triangle^2, \ \psi]  e_\lambda\|,
\end{align*}
where we have used the equation (\ref{last}) and $[\triangle^2, \ \psi]$ is a third order differential operator on $e_\lambda$ which involves the derivative of $\psi$. From the properties of $\psi$ and the fact that $\tau>1$, we have that
\begin{align*}
 \| r^{\epsilon} e^{\tau \phi} e_\lambda\|_{\frac{R}{2}, \frac{2R}{3}}+  \|  e^{\tau \phi} e_\lambda\|_{\frac{3\delta}{2}, 4\delta}
 &\leq C (\| e^{\tau \phi} e_\lambda\|_{\delta, \frac{3\delta}{2}}+\| e^{\tau \phi} e_\lambda\|_{R, 2R} ) \\
&+ C(\sum^{3}_{|\alpha|=1}\| r^{|\alpha|} e^{\tau \phi} \nabla^\alpha  e_\lambda\|_{\delta, \frac{3\delta}{2}}+\sum^{3}_{|\alpha|=1}\|  r^{|\alpha|} e^{\tau \phi} \nabla^\alpha e_\lambda\|_{R, 2R}),
\end{align*}
where the norm $\|\cdot\|_{R_1, R_2}=\|\cdot\|_{L^2( A_{R_1, R_2})}$.
Using the fact that the weight function $\phi$ is radial and decreasing, we could take the exponential function $e^{\tau\phi}$ out in these terms. We arrive at
\begin{align*}
e^{\tau \phi(\frac{2R}{3})}\|  e_\lambda\|_{\frac{R}{2}, \frac{2R}{3}}+ e^{\tau \phi({4\delta})}  \|  e_\lambda\|_{\frac{3\delta}{2}, 4\delta}
&\leq C (e^{\tau \phi(\delta) }\|  e_\lambda\|_{\delta, \frac{3\delta}{2}}+e^{\tau \phi(R) }\| e^{\tau \phi} e_\lambda\|_{R, 2R} ) \\
&+ C(e^{\tau \phi(\delta)}\sum^{3}_{|\alpha|=1} \| r^{|\alpha|} \nabla^\alpha e_\lambda\|_{\delta,
\frac{3\delta}{2}}+e^{\tau\phi(R)}\sum^{3}_{|\alpha|=1}\|r^{|\alpha|}  \nabla^\alpha e_\lambda\|_{R, 2R}).
\end{align*}
The use of  Caccioppoli type inequality for biharmonic equations implies that
\begin{align}
e^{\tau \phi(\frac{2R}{3})}\| e_\lambda\|_{\frac{R}{2}, \frac{2R}{3}}+ e^{\tau \phi({4\delta})}  \|   e_\lambda\|_{\frac{3\delta}{2}, 4\delta}
&\leq C  (e^{\tau \phi(\delta) }\|   e_\lambda\|_{2\delta}+e^{\tau\phi(R)}\| e^{\tau \phi} e_\lambda\|_{3R}).
\end{align}
See the Caccioppoli type inequality in (\ref{hihcac}) below or Lemma 1 in \cite{Zh5}.
Adding $ e^{\tau \phi({4\delta})} \|e_\lambda\|_{\frac{3\delta}{2}}$ to both sides of last inequality, we get that
\begin{align}
e^{\tau \phi(\frac{2R}{3})}\|  e_\lambda\|_{\frac{R}{2}, \frac{2R}{3}}+ e^{\tau \phi({4\delta})}  \| e_\lambda\|_{4\delta}
&\leq C (e^{\tau \phi(\delta) }\|  e_\lambda\|_{2\delta}+e^{\tau\phi(R)}\| e^{\tau \phi} e_\lambda\|_{3R}).
\label{drop}
\end{align}
We want to incorporate the second term in the right hand side of the last inequality into the left hand side. To this end,
we choose $\tau$ such that
$$C e^{\tau\phi(R)}\|e_\lambda\|_{3R}\leq \frac{1}{2}e^{\tau\phi(\frac{2R}{3})} \|e_\lambda\|_{\frac{R}{2}, \frac{2R}{3}}.  $$
That is, at least
\begin{align} \tau\geq \frac{1}{\phi(\frac{2R}{3})-\phi(R)}\ln \frac{ 2C \|e_\lambda\|_{3R}}{ \|e_\lambda\|_{\frac{R}{2}, \frac{2R}{3}}}.
\label{tautau}
\end{align}
For such $\tau$, we obtain that
\begin{equation}
e^{\tau \phi(\frac{2R}{3})}\| e_\lambda\|_{\frac{R}{2}, \frac{2R}{3}}+ e^{\tau \phi({4\delta})}  \|  e_\lambda\|_{4\delta}
\leq C e^{\tau \phi(\delta) }\|  e_\lambda\|_{2\delta}.
\label{droppd}
\end{equation}

To apply the Carleman estimates (\ref{Carle}), the assumption  that  $\tau\geq C $ for some $C$ independent of $\lambda$ is needed. In addition to the assumption (\ref{tautau}), we select
$$\tau=C+ \frac{1}{\phi(\frac{2R}{3})-\phi(R)}\ln \frac{ 2C \| e_\lambda\|_{3R}}{ \| e_\lambda\|_{\frac{R}{2}, \frac{2R}{3}}}. $$
Dropping the first term in (\ref{droppd}) gives that
\begin{align}
\|e_\lambda\|_{4\delta}&\leq C \exp\{ \big(C+ \frac{1}{\phi(\frac{2R}{3})-\phi(R)}\ln \frac{ 2C \|e_\lambda\|_{3R}}{ \|e_\lambda\|_{\frac{R}{2}, \frac{2R}{3}}}\big)\big(\phi(\delta)-\phi(4\delta)\big)  \}\|e_\lambda\|_{2\delta} \nonumber \\
&\leq {C }   (\frac{\|e_\lambda\|_{3R}}{ \|e_\lambda\|_{\frac{R}{2}, \frac{2R}{3}}})^C \|e_\lambda\|_{2\delta},
\label{tata}
\end{align}
where we have used the fact that
$$ \beta_1^{-1}<\phi(\frac{2R}{3})-\phi(R)<\beta_1,  $$
$$  \beta_2^{-1}< \phi(\delta)-\phi(4\delta)<\beta_2 $$
for some positive constants $\beta_1$ and $\beta_2$ that do not depend on $R$ or $\delta$. Since $\mathbb B_{3R}(x_0)\subset \tilde{\Omega}$,
it follows from (\ref{exdouex}) that
$$\frac{\|e_\lambda\|_{3R}}{ \|e_\lambda\|_{\frac{R}{2}, \frac{2R}{3}}}\leq e^{C\lambda}. $$
Thanks to the last inequality and (\ref{tata}), since $R$ is fixed, we derive that
\begin{align}\|e_\lambda\|_{4\delta}\leq  e^{C\lambda} \|e_\lambda\|_{2\delta} \end{align}
for some $C$ depending only on $\Omega$.
Let $\delta=\frac{r}{2}$. The doubling inequality
\begin{equation}
\|e_\lambda\|_{2r}\leq  e^{C\lambda} \|e_\lambda\|_{r}
\label{doub1}
\end{equation}
follows for $r\leq \frac{R}{12}$. If $\frac{R}{12}\leq r\leq \frac{d}{4}$, using (\ref{claim}), we can show that
\begin{align}
\|e_\lambda\|_{2r}&\geq \|e_\lambda\|_{\frac{R}{6}} \nonumber \\
&\geq e^{C(R)\lambda}\|e_\lambda\|_{\Omega_2}\nonumber \\
&\geq e^{C\lambda}\|e_\lambda\|_{r}.
\label{doub2}
\end{align}
Together with (\ref{doub1}) and (\ref{doub2}), we derive that
\begin{equation}
\|e_\lambda\|_{{2r}}\leq  e^{C\lambda} \|e_\lambda\|_{{r}}
\end{equation}
for any $0<r\leq \frac{d}{4}$ and $x_0\in \overline{\Omega}$, where $C$ only depends on the $\partial\Omega$. By standard elliptic estimates, the $L^\infty$ norm of doubling inequalities follows.
\end{proof}

An easy consequence of the doubling inequality (\ref{LLL}) is a vanishing order estimate for eigenfunctions $e_\lambda$ in $\Omega$.
\begin{corollary}
Let $e_\lambda$ be the  eigenfunction in (\ref{bistek1}), (\ref{bistek2}) or (\ref{bistek3}). Then the vanishing order of solution $e_\lambda$ in $\Omega$ is everywhere less than
$C\lambda$, where $C$ depends only on the real analytic domain $\Omega$.
\label{cor1}
\end{corollary}
\begin{proof}
The proof of the Corollary follows from the arguments in Corollary 1 in \cite{Zh4}. For the completeness of the presentation, we present the proof.
 We may assume that $\|e_\lambda\|_{L^\infty(\Omega)}=1$. Hence there exists some point $\bar x$ such that $\|e_\lambda\|_{L^\infty(\Omega)}=|e_\lambda(\bar x)|=1$. For any point $x_0\in \Omega$ and any $r>0$,
we iterate the doubling inequality (\ref{LLL}) $\hat{n}$ times so that
\begin{align}  2^{\hat{n}}r\leq d \label{distance}\end{align}
and \begin{align}
 \|e_\lambda\|_{L^\infty(\mathbb B_r(x_0))} \geq e^{-C\hat{n}\lambda}  \|e_\lambda\|_{L^\infty(\mathbb B_{2^{\hat{n}}r}(x_0))}.
 \label{mmb1}
 \end{align}
 Note that $\mathbb B_{2^{\hat{n}}r}(x_0)\subset\widetilde{ \Omega}$ and $d$ depends only on $\Omega$.
 Next we choose $x_1\in \partial \mathbb B_{(2^{\hat{n}}-1)r}(x_0)$ at $x_1$. It holds that
 \begin{align}
  \|e_\lambda\|_{L^\infty(\mathbb B_{(2^{\hat{n}}-1)r}(x_0))}\geq
 \|e_\lambda\|_{L^\infty(\mathbb B_r(x_1))}.
 \label{ccrol}
 \end{align}
 We also iterate the  doubling inequality (\ref{LLL}) $\hat{n}$ times. Thus,
 \begin{align}
 \|e_\lambda\|_{L^\infty(\mathbb B_{r}(x_1))}\geq e^{-C\hat{n}\lambda}
 \|e_\lambda\|_{L^\infty(\mathbb B_{2^{\hat{n}}r}(x_1))}.
 \label{mmb2}
 \end{align}
 After a finite number of steps, e.g. $m$ steps, we can arrive at $\bar x$. That is, $\bar x\in \mathbb B_{2^{\hat{n}}r}(x_{m-1})$. We also have
 \begin{align} m 2^{\hat{n}}r\leq \diam (\Omega).\end{align} Because of (\ref{distance}), we may choose $m=\frac{2\diam (\Omega)}{d}$. From the $m$ steps of iterations as (\ref{mmb1}), (\ref{ccrol}), and (\ref{mmb2}), we obtain that
 \begin{align}
\|e_\lambda\|_{L^\infty(\mathbb B_r(x_0))}&\geq e^{-C\lambda m \hat{n}}\|e_\lambda\|_{L^\infty(B_{2^{\hat{n}}r}(x_{m-1}))} \nonumber \\
 &\geq e^{-C\lambda \frac{2\diam (\Omega)}{d} \log_2 \frac{d}{2r}}\nonumber \\
&\geq (C r)^{C\lambda},
\label{vani}
\end{align}
where $C$ depends on $\Omega$.
Hence the estimate (\ref{vani}) implies that the vanishing order of solution at $x_0$ is less than $C\lambda$. Since $x_0$ is an arbitrary point, we get such vanishing rate of $e_\lambda$ at every point in $\Omega$.
\end{proof}

Thanks to the doubling inequality (\ref{LLL}), we are able to show the upper bounds of the nodal sets for eigenfunctions in (\ref{bistek1}), (\ref{bistek2}) or (\ref{bistek3}).
We need a lemma concerning the growth of a complex analytic function
with the number of zeros, see e.g. Lemma 2.3.2 in \cite{HL}.
\begin{lemma}
Suppose $f: \mathcal{B}_1(0)\subset \mathbb{C}\to \mathbb{C}$ is an
analytic function satisfying
$$ f(0)=1\quad \mbox{and} \quad \sup_{ \mathcal{B}_1(0)}|f|\leq 2^N$$
for some positive constant $N$. Then for any $r\in (0, 1)$, there
holds
$$\sharp\{z\in\mathcal{B}_r(0): f(z)=0\}\leq cN        $$
where $c$ depends on $r$. Especially, for $r=\frac{1}{2}$, there
holds
$$\sharp\{z\in \mathcal{B}_{1/2}(0): f(z)=0\}\leq N.        $$
\label{wwhy}
\end{lemma}

The idea to derive upper bounds of the measure of nodal sets in the real analytic setting using doubling inequalities and the complex growth lemma is kind of standard, see e.g. the pioneering work \cite{DF}, \cite{Lin} and \cite{HL}.
\begin{proof}[Proof of Theorem 1]
For any point $p\in \overline{\Omega}$,
applying elliptic estimates for eigenfunctions in (\ref{last})  in a
small ball $\mathbb B_{{r}}(p)\subset \widetilde{\Omega}$ yields that
\begin{equation}
|\frac{ D^{\alpha}e_\lambda (p)}{{\alpha} !}|\leq
C^{| \alpha|}_1 r^{-| \alpha|} \|e_\lambda \|_{L^\infty},
\end{equation}
where $C_1>1$ depends on $\Omega$.
We may consider the point $p$ as the origin. Summing up a geometric
series implies that we can extend $e_\lambda(x)$ to be a holomorphic function
$e_\lambda(z)$ with $z\in\mathbb{C}^n$. Furthermore, it holds that
\begin{equation}
\sup_{|z|\leq \frac{r}{2C_1}}|e_\lambda(z)|\leq C_2 \sup_{|x|\leq
r}|e_\lambda(x)|
\end{equation}
with $C_2>1$.

With aid of the doubling inequality (\ref{LLL}) and rescaling arguments, we can achieve that
\begin{equation}
\sup_{|z|\leq 2r}|e_\lambda(z)|\leq e^{C\lambda} \sup_{|x|\leq
r}|e_\lambda(x)| \label{dara13}
\end{equation}
for $0<r<r_0$ with $r_0$ depending on $\Omega$ and $C$ independent of $\lambda$ and $r$.

 We make use of Lemma \ref{wwhy} and the inequality (\ref{dara13}) to obtain the upper bounds of nodal sets for $e_\lambda(x)$. By rescaling and translation, we argue on scales
of order one. Let $p\in \mathbb B_{1/4}$ be the point where the
maximum of $|e_\lambda|$ in $\mathbb B_{1/4}$ is achieved. For each
direction $\omega \in S^{n-1}$, let $e_{\omega}(z)=
e_\lambda(p+z\omega)$ in $z\in \mathcal{B}_1(0)\subset\mathbb{C}$. Denote $N(\omega)=\sharp\{z\in\mathcal{B}_{1/2}(0)\subset
\mathbb{C}| e_\omega(z)=0\}$.
With aid of the doubling property (\ref{dara13}) and the Lemma \ref{wwhy}, we have that
\begin{eqnarray}
\sharp\{ x \in \mathbb B_{1/2}(p) &|& x-p \ \mbox{is parallel to} \
\omega \ \mbox{and} \ e_\lambda(x)=0\} \nonumber\\&\leq&
\sharp\{z\in\mathcal{B}_{1/2}(0)\subset
\mathbb{C}| e_\omega(z)=0\} \nonumber\\
&=&N(\omega)\nonumber\\ &\leq& C \lambda.
\end{eqnarray}
Thanks to the integral geometry estimates, we obtain that
\begin{eqnarray}
H^{n-1}\{ x \in \mathbb B_{1/2}(p)| e_\lambda(x)=0\} &\leq&
c(n)\int_{S^{n-1}} N(\omega)\,d \omega \nonumber\\
&\leq &\int_{S^{n-1}}C\lambda\, d\omega\nonumber\\ &=&C\lambda.
\end{eqnarray}
That is,
\begin{eqnarray}
H^{n-1}\{ x \in \mathbb B_{1/4}(0)| e_\lambda(x)=0\} \leq C\lambda.
\end{eqnarray}
 Since $\overline{\Omega}$ is compact, covering the domain $\overline{\Omega}$ using finitely many of balls gives that
\begin{equation} H^{n-1}\{x\in
\Omega|e_\lambda(x)=0\}\leq
C\lambda.
\end{equation}
Thus, we arrive at the
conclusion in Theorem \ref{th1}.

\end{proof}

\section{Nodal sets of eigenfunctions for buckling problems}
In this section, we aim to obtain the upper bounds of nodal sets of eigenfunctions for the buckling eigenvalue problem (\ref{bilap}). First of all, we need to analytically extend $e_\lambda$ across the boundary $\partial\Omega$.
The same arguments as the proof of Proposition 1 follow. We perform a lifting argument as $\hat{u}(x, t)=e^{\sqrt{\lambda} t} e_\lambda(x)$. Then $\hat{u}(x, t)$ satisfies the equation
\begin{equation}
\left \{ \begin{array}{lll}
\triangle^2 \hat{u}+ \partial_t^2\triangle\hat{u}=0  \quad &\mbox{in} \ {\Omega}\times (-\infty, \ \infty),  \medskip\\
\hat{u}=\frac{\partial \hat{u}}{\partial \nu}=0    \quad &\mbox{on} \ {\partial\Omega}\times (-\infty, \ \infty).
\end{array}
\right.
\end{equation}
Furthermore, let $$u(x, t, s)= e^{i \sqrt{\lambda} s}\hat{u}(x, t).$$ We have
\begin{equation}
\left \{ \begin{array}{lll}
\triangle^2 {u}+ \partial_t^2\triangle{u}+\partial^4_t u+\partial^4_s u=0  \quad &\mbox{in} \ {\Omega}\times (-\infty, \ \infty)\times (-\infty, \ \infty),  \medskip\\
{u}=\frac{\partial {u}}{\partial \nu}=0    \quad &\mbox{on} \ {\partial\Omega}\times (-\infty, \ \infty)\times (-\infty, \ \infty).
\label{die}
\end{array}
\right.
\end{equation}

Using the elliptic estimates in \cite{Mo} and \cite{MN}, we can extend the eigenfunction $u(x,t)$ across the boundary $\partial\Omega\times [-1, \ 1]\times [-1, \ 1]$ satisfying
\begin{align}
\triangle^2 {u}+ \partial_t^2\triangle{u}+\partial^4_t u+\partial^4_s u=0  \quad \mbox{in} \ {\widetilde{\Omega}}\times [-1, \ 1]\times [-1, \ 1]
\label{ccaa}
\end{align}
with
\begin{align}
\|u\|_{L^\infty(\widetilde{\Omega}\times [-1, \ 1]\times [-1, \ 1]}\leq C \|u\|_{L^\infty(\Omega\times [-2, \ 2]\times [-2, \ 2])},
\end{align}
where $\widetilde{\Omega}=\{ x\in \mathbb R^n|\dist(x, \Omega)\leq d\}$ and  $C$ depends only on $\Omega$. The uniqueness of the analytic continuation yields that
\begin{align}
\triangle^2 e_\lambda+\lambda \triangle e_\lambda=0  \quad &\mbox{in} \ {\widetilde{\Omega}}.
\label{biexten}
\end{align}
 Thus,
the growth of $e_\lambda$ can be controlled as
\begin{align}
\|e_\lambda\|_{L^\infty(\widetilde{\Omega})}\leq e^{C\sqrt{\lambda}}\|e_\lambda\|_{L^\infty(\Omega)}.
\label{want1}
\end{align}

Next we need to show three-ball inequalities, and then doubling inequalities for $e_\lambda$ in (\ref{biexten}) for any point $x_0\in \Omega$. To this end, we will establish the quantitative Carleman estimates for the operators in (\ref{biexten}).  The following quantitative Carleman estimates hold for Laplace eigenvalue problems, see e.g. \cite{DF}, \cite{BC} and \cite{Zh}. Let $\phi=-\ln r(x)+r^\epsilon(x)$ for some small constant $0<\epsilon<1$.  There exist positive constants $R_0$ and $C$, such that, for any $f\in C^\infty_0(\mathbb B_{R_0}(x_0)\backslash \mathbb B_{\delta}(x_0))$ and  $\tau>C(1+\sqrt{\|V(x)||_{C^1}})$, one has
 \begin{align}
 C\| r^2 e^{\tau \phi} (\triangle+V(x)) f\|&\geq \tau^{\frac{3}{2}}\|r^{\frac{\epsilon}{2}}e^{\tau\phi} f\|+\tau \delta \|r^{-1} e^{\tau\phi} f\|  \nonumber \\ &+
 \tau^{\frac{1}{2}}\|r^{1+\frac{\epsilon}{2}}e^{\tau\phi} \nabla f\|.
 \label{similar}
\end{align}
See corollary 2.2 in \cite{BC} for (\ref{similar}).
We iterate (\ref{similar}) to derive the quantitative Carleman estimates for the operator in (\ref{biexten}) as follows.
\begin{lemma}
There exist positive constants $R_0$ and $C$,  such that, for any $x_0\in \Omega$, any $f\in C^\infty_0(\mathbb B_{R_0}(x_0)\backslash \mathbb B_{\delta}(x_0))$ with $0<\delta<R_0<\frac{d}{10}<1$,  and  $\tau>C(1+\sqrt{\lambda})$, one has
\begin{align}
C\|r^4 e^{\tau \phi} (\triangle^2 f+\lambda \triangle f) \|\geq \tau^3 \|r^\epsilon e^{\tau \phi} f\|+\tau^2 \delta^2 \|r^{-2} e^{\tau \phi} f\|.
\label{Carle1}
\end{align}
\label{prob}
\end{lemma}

\begin{proof}
The definition of the weight function $\phi=-\ln r+r^\epsilon$ gives that
$$  r^4 e^{\tau\phi}=r^2 e^{(\tau-2)\phi}e^{2 r^\epsilon}. $$
Since $0<r<R_0<1$, then $1<e^{2 r^\epsilon}<e^2$. It follows from (\ref{similar}) that, for $\tau>C(1+\sqrt{\lambda})$,
\begin{align}
C\| r^4 e^{\tau \phi} (\triangle+\lambda)\triangle f \|&\geq  C\| r^2 e^{(\tau-2) \phi}  (\triangle+\lambda)\triangle f  \| \nonumber\\
&\geq \tau^{\frac{3}{2}}\|r^{\frac{\epsilon}{2}} e^{(\tau-2)\phi} \triangle f\|.
\label{hih1}
\end{align}
Elementary calculations show that
\begin{align}r^{\frac{\epsilon}{2}}e^{(\tau-2)\phi}&=r^2 e^{\tau\phi}r^{\frac{\epsilon}{2}} e^{-2 r^\epsilon}\nonumber \\
&= r^2 e^{ (\tau-\frac{\epsilon}{2})\phi  } e^{\frac{\epsilon}{2} r^\epsilon} e^{-2 r^\epsilon}. \end{align}
It follows that
$$|r^{\frac{\epsilon}{2}}e^{(\tau-2)\phi}|\geq C r^2  e^{(\tau-\frac{\epsilon}{2})\phi}.  $$
Thanks to (\ref{similar}) again, we obtain that
\begin{align}
\|r^{\frac{\epsilon}{2}}e^{(\tau-2)\phi} \triangle f\|&\geq C\|r^2  e^{(\tau-\frac{\epsilon}{2})\phi}\triangle f\|\nonumber\\
 &\geq C\tau^{\frac{3}{2}} \|r^{\frac{\epsilon}{2}} e^{(\tau-\frac{\epsilon}{2})\phi} f\| \nonumber \\
&\geq C\tau^{\frac{3}{2}} \|r^\epsilon e^{\tau\phi} f\|,
\label{hih2}
\end{align}
where the following estimate is used
$$ e^{-\frac{\epsilon}{2}\phi}=r^{\frac{\epsilon}{2}} e^{-r^\epsilon}\geq r^{\frac{\epsilon}{2}} e^{-1}.   $$
Together with the inequalities (\ref{hih1}) and (\ref{hih2}), we arrive at
\begin{equation}
\| r^4 e^{\tau \phi} (\triangle+\lambda) \triangle f \|\geq C\tau^3 \|r^\epsilon e^{\tau\phi} f\|.
\label{nana}
\end{equation}
Applying the similar argument as the way in showing (\ref{nana}), we can derive that
\begin{equation}
\| r^4 e^{\tau \phi} (\triangle+\lambda) \triangle f \|\geq C\tau^2\delta^2 \|r^{-2} e^{\tau\phi} f\|.
\label{baba}
\end{equation}
In view of (\ref{nana}) and (\ref{baba}), we obtain the desired estimates
\begin{align}
 C\| r^4 e^{\tau \phi} (\triangle^2 f+\lambda \triangle f)\|\geq \tau^3\|r^{\epsilon}e^{\tau\phi} f\|+\tau^2 \delta^2 \|r^{-2} e^{\tau\phi} f\|.
 \label{sasa}
 \end{align}
\end{proof}

As the arguments  in Section 2, the three-ball inequalities are important tools in characterizing the growth of eigenfunctions.
We employ Carleman estimates (\ref{Carle1}) to show  the three-ball inequalities for the solution $e_\lambda$ in (\ref{biexten}).
\begin{lemma}
There exist positive constants $R_0$, $C$ and $0<\beta<1$  such that, for any $R< \frac{R_0}{8}$ and any $x_0\in \Omega$, the solutions $e_\lambda$ of (\ref{biexten}) satisfy
\begin{equation}
\|e_\lambda\|_{\mathbb B_{2R}(x_0)}\leq e^{C \sqrt{\lambda}} \|e_\lambda\|^{\beta}_{\mathbb B_R(x_0)} \|e_\lambda\|^{1-\beta}_{\mathbb B_{3R}(x_0)}.
\label{balls}
\end{equation}
\label{baa}
\end{lemma}

\begin{proof}
We introduce a smooth cut-off function $0<\psi(r)<1$  satisfying the following properties:
\begin{itemize}
\item $\psi(r)=0$ \ \ \mbox{if} \ $r(x)<\frac{R}{4}$ \ \mbox{or} \  $r(x)>\frac{5R}{2}$, \medskip
\item $\psi(r)=1$ \ \ \mbox{if} \ $\frac{3R}{4}<r(x)<\frac{9R}{4}$, \medskip
\item $|\nabla^\alpha \psi|\leq \frac{C}{R^{|\alpha|}}$
\end{itemize}
for $R<
\frac{R_0}{8}$. Since the function $\psi u$ is supported in the annulus $A_{\frac{R}{4}, \frac{5R}{2}}$, applying the Carelman estimates (\ref{Carle1}) with $f=\psi e_\lambda$, we derive that
\begin{align}
\tau^2 \| e^{\tau \phi}e_\lambda\|_{\frac{3R}{4}, \frac{9R}{4}}& \leq C\| r^4 e^{\tau \phi}\big(\triangle^2 (\psi e_\lambda)+\lambda \triangle (\psi e_\lambda)\big)\| \nonumber  \\
&=C \| r^4 e^{\tau \phi}([\triangle^2, \ \psi] e_\lambda+\lambda \triangle \psi e_\lambda+2\lambda \nabla \psi\nabla e_\lambda)\|,
\end{align}
where we have used the equation (\ref{biexten}). Note that $[\triangle^2, \ \psi]$ is a third order differential operator on $e_\lambda$ involving the derivative of $\psi$.
By the properties of $\psi$, we get that
\begin{align*}
\| e^{\tau \phi} e_\lambda\|_{\frac{3R}{4}, \frac{9R}{4}}&\leq C\lambda (\| e^{\tau \phi} e_\lambda\|_{\frac{R}{4}, \frac{3R}{4}}+\| e^{\tau \phi} e_\lambda\|_{\frac{9R}{4}, \frac{5R}{2}} ) \\
&+ C(\sum_{|\alpha|=1}^{3} \| r^{|\alpha|} e^{\tau \phi} \nabla^\alpha e_\lambda \|_{\frac{R}{4}, \frac{3R}{4}}+\sum_{|\alpha|=1}^{3} \| r^{|\alpha|} e^{\tau \phi} \nabla^\alpha e_\lambda \|_{\frac{9R}{4}, \frac{5R}{2}}) \\
&+C\lambda(\| r^3 e^{\tau \phi} \nabla e_\lambda \|_{\frac{R}{4}, \frac{3R}{4}}+ \| r^3 e^{\tau \phi} \nabla e_\lambda \|_{\frac{9R}{4}, \frac{5R}{2}}).
\end{align*}
Since the weight function $\phi$ is radial and decreasing, we obtain that
\begin{align}
 \|e^{\tau \phi}  e_\lambda\|_{\frac{3R}{4}, \frac{9R}{4}}&\leq C\lambda (e^{\tau \phi(\frac{R}{4})} \| e_\lambda\|_{\frac{R}{4}, \frac{3R}{4}}+e^{\tau \phi(\frac{9R}{4})} \| e_\lambda\|_{\frac{9R}{4}, \frac{5R}{2}} ) \nonumber \\
&+ C(e^{\tau \phi(\frac{R}{4})}\sum_{|\alpha|=1}^{3}\| r^{|\alpha|} \nabla^\alpha e_\lambda\|_{\frac{R}{4}, \frac{3R}{4}}+e^{\tau \phi(\frac{9R}{4})}\sum_{|\alpha|=1}^{3}\| r^{|\alpha|} \nabla^\alpha e_\lambda\|_{\frac{9R}{4}, \frac{5R}{2}})\nonumber \\
&+C\lambda(e^{\tau \phi(\frac{R}{4})} \| r^3  \nabla e_\lambda \|_{\frac{R}{4}, \frac{3R}{4}}+ e^{\tau \phi(\frac{9R}{4})}\| r^3  \nabla e_\lambda \|_{\frac{9R}{4}, \frac{5R}{2}}).
\label{recall}
\end{align}
For the higher order elliptic equations
\begin{equation}  \triangle^2 u+\lambda \triangle u=0,
\label{zhu}
\end{equation}
the following Caccioppoli type inequality holds
\begin{equation}
\sum^{3}_{|\alpha|=0} \|r^{|\alpha|}\nabla^\alpha u\|_{c_3R, \ c_2R} \leq C (\lambda+1)^{3} \|u\|_{c_4R, \  c_1R}
\label{hihcac}
\end{equation}
 for all positive constants $0<c_4<c_3<c_2<c_1<1$. See e.g. Lemma 1 in \cite{Zh2} for such quantitative Caccioppoli type inequality.
It follows from (\ref{hihcac}) that
\begin{align*}
\| r^{|\alpha|} \nabla^\alpha e_\lambda\|_{\frac{R}{4}, \frac{3R}{4}}\leq C\lambda^3 \| e_\lambda\|_{R}
\end{align*}
and
\begin{align*}
\| r^{|\alpha|} \nabla^\alpha e_\lambda\|_{\frac{9R}{4}, \frac{5R}{2}}\leq  C \lambda^3 \| e_\lambda\|_{3R}
\end{align*}
for all $1\leq |\alpha|\leq 3$ and $\lambda\geq 1$. Thus, the estimate (\ref{recall}) yields that
\begin{align}
\|e_\lambda\|_{\frac{3R}{4}, 2R}\leq C\lambda^{4}\big( e^{\tau(\phi(\frac{R}{4})-\phi(2R))} \|e_\lambda\|_R+  e^{\tau(\phi(\frac{9R}{4})-\phi(2R))} \|e_\lambda\|_{3R}\big).
\label{inequal}
\end{align}
We choose parameters
$$ \beta^1_R=\phi(\frac{R}{4})-\phi(2R),$$
$$ \beta^2_R=\phi(2R)-\phi(\frac{9R}{4}).$$
From the definition of the weight function $\phi$, it holds that
$$ 0<\beta^{-1}_1<\beta^1_R<\beta_1 \quad \mbox{and} \quad 0<\beta_2<\beta^2_R<\beta^{-1}_2,$$
where $\beta_1$ and $\beta_2$ independent of $R$.
Adding $\|e_\lambda\|_{\frac{3R}{4}}$ to both sides of the inequality (\ref{inequal}) leads that
\begin{equation}
\|e_\lambda\|_{2R}\leq C\lambda^{4}\big( e^{\tau\beta_1}\|e_\lambda\|_R+  e^{-\tau\beta_2}\|e_\lambda\|_{3R}     \big).
\end{equation}
To incorporate the second term in the right hand side of the last inequality into the left hand side,
 we choose $\tau$ such that
\begin{align}C\lambda^{4}e^{-\tau\beta_2}\|e_\lambda\|_{3R}\leq \frac{1}{2}\|e_\lambda\|_{2R}. \label{incor}  \end{align}
The inequality (\ref{incor}) holds if
$$\tau\geq \frac{1}{\beta_2} \ln \frac{2C\lambda^{4}\|e_\lambda\|_{3R}}{\|e_\lambda\|_{2R} }.   $$
Thus, for such $\tau$, we obtain that
\begin{equation}
\|e_\lambda\|_{2R}\leq C\lambda^{4} e^{\tau\beta_1}\|e_\lambda\|_R.
\label{substitute}
\end{equation}
Since $\tau>C\sqrt{\lambda}$ is needed to apply the Carleman estimates (\ref{Carle1}), we select
$$ \tau=C\sqrt{\lambda}+\frac{1}{\beta_2} \ln \frac{2C\lambda^{4}\|e_\lambda\|_{3R}}{\|e_\lambda\|_{2R} }. $$
Substituting such $\tau$ in (\ref{substitute}) gives that
\begin{align}
\|e_\lambda\|_{2R}^{\frac{\beta_2+\beta_1}{\beta_2}} \leq e^{ C\sqrt{\lambda}}\|e_\lambda\|_{3R}^{\frac{\beta_1}{\beta_2}} \|e_\lambda\|_R.
\end{align}
Raising the exponent $\frac{\beta_2}{\beta_2+\beta_1}$ to both sides of the last inequality yields that
\begin{align}
\|e_\lambda\|_{2R} \leq e^{ C\sqrt{\lambda}}\|e_\lambda\|_{3R}^{\frac{\beta_1}{\beta_1+\beta_2}} \|e_\lambda\|_R^{\frac{\beta_2}{\beta_1+\beta_2}}.
\end{align}
Set $\beta={\frac{\beta_2}{\beta_1+\beta_2}}$. Then $0<\beta<1$. We arrive at the three-ball inequality in the lemma.
\end{proof}

Using the three-ball inequality (\ref{balls}) and growth of $e_\lambda$ estimates (\ref{want1}), following the proof of (\ref{exdouex}) in Section 2, we can show an analogous estimate
\begin{align}
\|e_\lambda\|_{L^\infty( A_{\frac{R}{2}, \ \frac{3R}{2}}({x_0}))}
&\geq e^{-{C(R){\sqrt{\lambda}}}} \|e_\lambda\|_{L^\infty(\tilde{\Omega})}.
\end{align}

Following the argument in the proof of Proposition 2 and  applying the Carleman estimates in (\ref{Carle1}) for eigenfunctions in (\ref{biexten}), we are able to derive the following doubling inequalities
\begin{equation}
\|e_\lambda\|_{L^\infty(\mathbb B_{2r}(x_0))}\leq e^{C\sqrt{\lambda}}\|e_\lambda\|_{L^\infty(\mathbb B_{r}(x_0))}
\label{NNN}
\end{equation}
for any $x_0\in \overline{\Omega}$ and $0<r\leq \frac{d}{4}$. By Corollary 1, the doubling inequality (\ref{NNN}) readily implies that the vanishing order of $e_\lambda$ is everywhere less than $C\sqrt{\lambda}$ in $\Omega$.

The proof of Theorem \ref{th2} is derived  using the doubling inequalities (\ref{NNN}) and Lemma \ref{wwhy} as the arguments in Theorem \ref{th1}.

\begin{proof}[Proof of Theorem 2]
For any point $(p, 0, 0)\in \overline{\Omega}\times [-\frac{1}{2}, \frac{1}{2}]\times [-\frac{1}{2}, \frac{1}{2}]$,
applying elliptic estimates for $u(x, t, s)$ in (\ref{ccaa}) in a
small ball $\mathbb B_{{r}}(p)\times [-r, r]\times [-r, r]\subset \widetilde{{\Omega}}\times [-1, 1\times [-1, 1]$, we have
\begin{equation}
|\frac{ D^{\alpha}_x u(p, 0, 0)}{{\alpha} !}|\leq
C^{| \alpha|}_3  r^{-| \alpha|} \|u\|_{L^\infty},
\end{equation}
where  $D^{\alpha}_x$ is the $|\alpha|$-order partial derivatives with respect to $x$ and $C_3>1$ depends on $\Omega$. By translation, we consider the point $p$ as the origin. From the definition of $u$, we obtain that
\begin{equation}
|\frac{ D^{\alpha} e_\lambda(0)}{{\alpha} !}|\leq
C^{| \alpha|}_3  r^{-| \alpha|} e^{C\sqrt{\lambda}}\|e_\lambda\|_{L^\infty(\mathbb B_r)}.
\end{equation}
 Thus, $e_\lambda(x)$ can be extended to be a holomorphic function
$e_\lambda(z)$ with $z\in\mathbb{C}^n$ by summing up a geometric series to have
\begin{align}
\sup_{|z|\leq \frac{r}{2C_3}}|e_\lambda(z)|\leq  e^{C_4\sqrt{\lambda}}\sup_{|x|\leq
r}|e_\lambda(x)|
\end{align}
for $C_4>1$.
Taking advantage of the doubling inequality (\ref{NNN}), from rescaling arguments,  we arrive at
\begin{equation}
\sup_{|z|\leq 2r}|e_\lambda(z)|\leq e^{C\sqrt{\lambda}} \sup_{|x|\leq
r}|e_\lambda(x)| \label{daranew}
\end{equation}
for $0<r<r_0$ with $r_0$ depending on $\Omega$ and $C$ independent of $r$ and $\lambda$.

We combine  Lemma \ref{wwhy} and the estimates
(\ref{daranew}) to obtain the measure of nodal sets.  By rescaling and translation, we  argue on scales
of order one. Let $p\in \mathbb B_{1/4}$ be  the
maximum of $|e_\lambda|$ in $\mathbb B_{1/4}$. For each
direction $\omega \in S^{n-1}$, let $e_{\omega}(z)=
e_\lambda(p+z\omega)$ in $z\in \mathcal{B}_1(0)\subset\mathbb{C}$. Recall that $N(\omega)=\sharp\{z\in\mathcal{B}_{1/2}(0)\subset
\mathbb{C}| e_\omega(z)=0\}$.
Applying the doubling property (\ref{daranew}) and the Lemma \ref{wwhy}, we have that
\begin{eqnarray}
\sharp\{ x \in \mathbb B_{1/2}(p) &|& x-p \ \mbox{is parallel to} \
\omega \ \mbox{and} \ e_\lambda(x)=0\} \nonumber\\&\leq&
\sharp\{z\in\mathcal{B}_{1/2}(0)\subset
\mathbb{C}| e_\omega(z)=0\} \nonumber\\
&=&N(\omega)\nonumber\\ &\leq& C\sqrt{\lambda}.
\end{eqnarray}
From the integral geometry estimates, we can show that
\begin{eqnarray}
H^{n-1}\{ x \in \mathbb B_{1/2}(p)| e_\lambda(x)=0\} &\leq&
c(n)\int_{S^{n-1}} N(\omega)\,d \omega \nonumber\\
&\leq &\int_{S^{n-1}}C\sqrt{\lambda}\, d\omega\nonumber\\ &=&C\sqrt{\lambda}.
\end{eqnarray}
Hence, it follows that
\begin{eqnarray}
H^{n-1}\{ x \in \mathbb B_{1/4}(0)| e_\lambda(x)=0\} \leq C\sqrt{\lambda}.
\end{eqnarray}
 Covering the domain $\overline{\Omega}$ using a finite number of balls yields that
\begin{equation} H^{n-1}\{x\in
\Omega|e_\lambda(x)=0\}\leq
C\sqrt{\lambda}.
\end{equation}
Therefore,  the proof in Theorem \ref{th2} is completed.

\end{proof}

\section{Nodal sets of eigenfunctions for champed-plate problems}
We are also interested in the upper bounds of nodal sets for the eigenvalue problem
\begin{equation}
\left \{ \begin{array}{lll}
\triangle^2 e_\lambda=\lambda e_\lambda  \quad &\mbox{in} \ {\Omega},  \medskip\\
\frac{\partial e_\lambda}{\partial \nu}=e_\lambda =0 \quad &\mbox{on} \ {\partial\Omega},
\end{array}
\right.
\label{stlap}
\end{equation}
which is the bi-Laplace eigenvalue problem with Dirichlet boundary conditions. As before, we aim to extend $e_\lambda$ across the boundary $\partial\Omega$ analytically.
We adopt the lifting argument.
Let
\begin{align}
u(x, t)= e^{ \sqrt{i}\lambda^\frac{1}{4}t} e_\lambda(x).
\label{defin}
\end{align}
It follows that
\begin{equation}
\left \{ \begin{array}{lll}
\triangle^2 u+\partial^4_t u=0 \quad &\mbox{in}  \ {\Omega}\times(-\infty, \ -\infty),  \medskip\\
\frac{\partial u}{\partial \nu}=u=0 \quad &\mbox{on} \ {\partial\Omega}\times(-\infty, \ -\infty).
\label{diee}
\end{array}
\right.
\end{equation}
Following the arguments in Proposition 1, we can extend $u(x, t)$ analytically across the boundary $\partial\Omega\times [-1, \ 1]$. Thus, $u(x, t)$ satisfies
\begin{align}
\triangle^2 u+\partial^4_t u=0 \quad &\mbox{in}  \ {\widetilde{\Omega}}\times [-1, \ 1]
\label{niucao}
\end{align}
 with
 \begin{align}
 \|u\|_{L^\infty(\widetilde{\Omega}\times [-1, \ 1])}\leq C \|u\|_{L^\infty(\Omega\times [-2, \ 2])},
 \end{align}
 where ${\widetilde{\Omega}}=\{x\in \mathbb R^n| \dist(x, \ \Omega)\leq d\}$ for some $d>0$ depends on $\partial\Omega$.
 The uniqueness of the analytic continuation gives that
\begin{align}
\triangle^2 e_\lambda=\lambda e_\lambda  \quad &\mbox{in} \ {\widetilde{\Omega}}.
\label{dark}
\end{align}
 From the definition of $u(x, t)$,
we will have the growth control estimates
\begin{align}
\|e_\lambda\|_{L^\infty(\widetilde{{\Omega}})}\leq e^{C \lambda^\frac{1}{4}} \|e_\lambda\|_{L^\infty({\Omega})},
\label{indu}
\end{align}

Next we need to establish some quantitative Carleman estimates to obtain doubling inequalities.
 The bi-Laplace operator with eigenvalues can be decomposed as
\begin{align}
\triangle^2-\lambda=(\triangle-\sqrt{\lambda})(\triangle+\sqrt{\lambda}),
\label{decom}
\end{align}
As the proof of Lemma \ref{prob},  we iterate the quantitative Carleman estimates (\ref{similar}) using the decomposition (\ref{decom}). There exist $R_0$ and $C$ as in Lemma \ref{prob} such that
\begin{align}
C\|r^4 e^{\tau \phi} (\triangle^2 f-\lambda f) \|&\geq \tau^\frac{3}{2} \|r^{\frac{\epsilon}{2}} e^{(\tau-2) \phi}(\triangle-\sqrt{\lambda}) f\|+\tau \delta \|r^{-1} e^{(\tau-2) \phi} (\triangle-\sqrt{\lambda})f\|\nonumber \\
&\geq \tau^3 \|r^{{\epsilon}} e^{\tau\phi} f\|+\tau^2 \delta^2 \|r^{-2} e^{\tau \phi} f\|
\label{newcarle}
\end{align}
for  any $f\in C^\infty_0(\mathbb B_{R_0}(x_0)\backslash \mathbb B_{\delta}(x_0))$ and  $\tau>C(1+\lambda^{\frac{1}{4}})$ with $\phi=-\ln r(x)+r^\epsilon(x)$.

By Carleman estimates (\ref{newcarle}) and the arguments in Lemma \ref{baa}, we can show the following three-ball inequality
\begin{align}
\|e_\lambda\|_{L^2(\mathbb B_{2R}(x_0))}\leq  e^{C\lambda^{\frac{1}{4}}} \|e_\lambda\|^\beta_{L^2(\mathbb B_{R}(x_0))}\|e_\lambda\|^{1-\beta}_{L^2(\mathbb B_{3R}(x_0))}.
\end{align}
By standard elliptic estimates, we have the $L^\infty$-norm three-ball inequality
\begin{align}
\|e_\lambda\|_{L^\infty(\mathbb B_{2R}(x_0))}\leq  e^{C\lambda^{\frac{1}{4}}}\|e_\lambda\|^\beta_{L^\infty(\mathbb B_{R}(x_0))}\|e_\lambda\|^{1-\beta}_{L^\infty(\mathbb B_{3R}(x_0))}.
\label{lastthree}
\end{align}

Following the arguments of (\ref{exdouex}) in Section 2 by applying the three-ball inequalities (\ref{lastthree}) finite times and (\ref{indu}), we obtain that
\begin{align}
\|e_\lambda\|_{L^\infty\big( A_{\frac{R}{2}, \ \frac{3R}{2}}({x_0})\big)}
&\geq e^{-{C(R)\lambda^{\frac{1}{4}}}} \|e_\lambda\|_{L^\infty(\widetilde{\Omega})}.
\label{exdou}
\end{align}

From the Carleman estimates (\ref{newcarle}), (\ref{exdou}),
 and the arguments in Proposition 2, we are able to derive the doubling inequality
\begin{align}
\|e_\lambda\|_{L^\infty\big(\mathbb B_{2r}(x_0)\big)}\leq e^{C\lambda^{\frac{1}{4}}} \|e_\lambda\|_{L^\infty\big(\mathbb B_{r}(x_0)\big)}
\label{fdoubl}
\end{align}
for any $0<r\leq\frac{d}{4}$ and $x_0\in \Omega$.

As in the proof of Theorem 1, we prove the upper bounds of nodal sets for eigenfunctions in (\ref{stlap}) using doubling inequalities (\ref{fdoubl}) and the complex growth Lemma \ref{wwhy}.

\begin{proof}[Proof of Theorem \ref{th3}]
For any point $(p, 0, 0) \in \overline{\Omega}\times [-\frac{1}{2}, \frac{1}{2}] $,
the elliptic estimates for $u(x, t)$ in (\ref{niucao}) in a
small ball $\mathbb B_{{r}}(p)\times (-r, r)\subset {\widetilde{\Omega}}\times [-1, 1]$ gives that
\begin{equation}
|\frac{ D_x^{ \alpha}u (p, 0)}{{\alpha} !}|\leq
C^{|\alpha|}_5  r^{-|\alpha|}\|u \|_{L^\infty},
\label{complex}
\end{equation}
where $C_5>1$ depends on $\Omega$.
We may consider the point $p$ as the origin as well. The definition of $u(x, t)$ yields that
 \begin{equation}
|\frac{ D^{ \alpha}e_\lambda (0)}{{\alpha} !}|\leq
C^{|\alpha|}_5  r^{-|\alpha|}e^{C\lambda^{\frac{1}{4}}}\|e_\lambda \|_{L^\infty(\mathbb B_r)}.
\end{equation}
 Thanks to (\ref{complex}),  we can sum up a geometric
series to extend $e_\lambda(x)$ to be a holomorphic function
$e_\lambda(z)$ with $z\in\mathbb{C}^{n}$. Then we have
\begin{align}
\sup_{|z|\leq \frac{r}{2C_5}}|e_\lambda(z)|\leq  e^{C_6\lambda^{\frac{1}{4}}}\sup_{|x|\leq
r}|e_\lambda(x)|
\end{align}
with $C_6>1$.
Thanks to the doubling inequality (\ref{fdoubl}), from rescaling arguments, it holds that
\begin{equation}
\sup_{|z|\leq 2r}|e_\lambda(z)|\leq e^{C\lambda^{\frac{1}{4}}}
\sup_{|x|\leq r}|e_\lambda(x)|
\label{dara1}
\end{equation}
for any $0<r<r_0$ with $r_0$ depending on $\Omega$ and $C$ independent of $r$ and $\lambda$.

Next we provide the proof of the upper bounds of nodal sets for $e_\lambda$.
  By rescaling and translation, we argue on scales
of order one. Let $p\in \mathbb B_{1/4}$ be the point where the
maximum of $|e_\lambda|$ in $\mathbb B_{1/4}$ is achieved. For each
direction $\omega \in S^{n-1}$, let $e_{\omega}(z)=
e_\lambda(p+z\omega)$ in $z\in \mathcal{B}_1(0)\subset\mathbb{C}$.
The doubling inequality (\ref{dara1}) and Lemma \ref{wwhy} yield that
\begin{eqnarray}
\sharp\{ x \in \mathbb B_{1/2}(p) &|& x-p \ \mbox{is parallel to} \
\omega \ \mbox{and} \ e_\lambda(x)=0\} \nonumber\\&\leq&
\sharp\{z\in\mathcal{B}_{1/2}(0)\subset
\mathbb{C}| e_\omega(z)=0\} \nonumber\\
&=&N(\omega)\nonumber\\ &\leq& C \lambda^\frac{1}{4}.
\end{eqnarray}
With aid of the integral geometry estimates, we derive that
\begin{eqnarray}
H^{n-1}\{ x \in \mathbb B_{1/2}(p)| e_\lambda(x)=0\} &\leq&
c(n)\int_{S^{n-1}} N(\omega)\,d \omega \nonumber\\
&\leq &\int_{S^{n-1}}C\lambda^\frac{1}{4}\, d\omega\nonumber\\ &=&C\lambda^\frac{1}{4},
\end{eqnarray}
which implies that
\begin{eqnarray}
H^{n-1}\{ x \in \mathbb B_{1/4}(0)| e_\lambda(x)=0\} \leq C\lambda^\frac{1}{4}.
\end{eqnarray}
 Covering the domain $\overline{\Omega}$ using finitely many of balls leads to
\begin{equation} H^{n-1}\{x\in
\Omega|e_\lambda(x)=0\}\leq
C\lambda^\frac{1}{4}. \label{last2}
\end{equation}
This completes the
conclusion in Theorem \ref{th3}. \end{proof}

For eigenvalue problems of higher order elliptic equations of general orders, two types of boundary conditions are commonly studied. There are higher order elliptic equations with Dirichlet boundary conditions
\begin{align}
\left \{\begin{array}{lll}
(-\triangle)^m e_\lambda=\lambda e_\lambda  \quad &\mbox{in} \ {\Omega},  \\
\frac{ \partial^{m-1} e_\lambda}{\partial \nu^{m-1}}=\frac{ \partial^{m-2} e_\lambda}{\partial \nu^{m-2}}=\cdots=e_\lambda=0 &\mbox{on} \ {\partial\Omega}
\end{array}
\right.
\label{moder}
\end{align}
and higher order elliptic equation with Navier boundary conditions
\begin{align}
\left \{\begin{array}{lll}
(-\triangle)^m e_\lambda=\lambda e_\lambda  \quad &\mbox{in} \ {\Omega},  \\
\triangle^{m-1} e_\lambda =\triangle^{m-2} e_\lambda=\cdots=e_\lambda=0 &\mbox{on} \ {\partial\Omega}
\end{array}
\right.
\label{Navier}
\end{align}
for any integer $m\geq 2$.
The approach in the proof of Theorem 3 is also applied for both (\ref{moder}) and (\ref{Navier}). We can obtain the following upper bounds of nodal sets.
\begin{corollary}
Let $e_\lambda$ be the eigenfunction in (\ref{moder}) or (\ref{Navier}). There exists a positive constant $C$ depending only on the real analytic domain $\Omega$ such that
\begin{equation} H^{n-1}(\{x\in \Omega |e_\lambda(x)=0\})\leq C\lambda^\frac{1}{2m}.
 \end{equation}
\end{corollary}
\begin{proof}
We only sketch the main ideas of the proof, since the arguments are quite similar to the proof of Theorem \ref{th3}.  We first consider the eigenvalue problem (\ref{moder}).
To do the analytic continuation across the boundary $\partial\Omega$,  we perform the lifting argument for $e_\lambda$ in (\ref{moder})
as
\begin{align*}
u(x, t)=\left \{ \begin{array}{lll}
e^{ i^\frac{1}{m} \lambda^{\frac{1}{2m}}t } e_\lambda \quad &m \ \mbox{even}, \\
e^{  \lambda^{\frac{1}{2m}}t } e_\lambda \quad &m \ \mbox{odd}.
\end{array}
\right.
\end{align*}
Then  $u(x, t)$ satisfies the equation
\begin{align}
\left \{\begin{array}{lll}
(-\triangle)^m u(x, t)+ (-1)^m\partial_t^{2m} u(x, t)=0  \quad &\mbox{in} \ {\Omega}\times (-\infty, \ \infty),  \\
\frac{ \partial^{m-1} u}{\partial \nu^{m-1}}=\frac{ \partial^{m-2} u}{\partial \nu^{m-2}}=\cdots=u=0 &\mbox{on} \ {\partial\Omega}\times (-\infty, \ \infty).
\end{array}
\right.
\end{align}

Following the arguments in Proposition 1, the elliptic estimates  will allow the analytic continuation of $u(x, t)$ across the boundary $\partial\Omega\times [-1, 1]$. Then we have
\begin{align}
(-\triangle)^m u+(-1)^m \partial^{2m}_t u=0 \quad &\mbox{in}  \ {\widetilde{\Omega}}\times [-1, \ 1].
\end{align}
 We are also able to  derive the equation
\begin{align}
(-\triangle)^m e_\lambda=\lambda e_\lambda  \quad &\mbox{in} \ {\widetilde{\Omega}}
\end{align}
and have the growth estimates
\begin{align}
\|e_\lambda\|_{L^\infty(\widetilde{{\Omega}})}\leq e^{C \lambda^\frac{1}{2m}} \|e_\lambda\|_{L^\infty({\Omega})}.
\label{higher}
\end{align}

Next step is to obtain the doubling inequalities.
We adapt the quantitative Carleman estimates (\ref{similar}) for the higher order elliptic operator $(-\triangle)^m-\lambda$.  By fundamental theorem of algebra, the higher order elliptic operator can be decomposed as
\begin{align}
(-\triangle)^m-\lambda=\prod^{m-1}_{k=0}(-\triangle- \lambda^{\frac{1}{m}} e^{\frac{2ki\pi}{m}}).
\end{align} We
iterate the Carleman estimates (\ref{similar}) $m$ times as Lemma \ref{prob}. It follows that
\begin{align}
C\|r^{2m} e^{\tau \phi} ((-\triangle)^m-\lambda) f\|\geq \tau^\frac{3m}{2} \|r^{\frac{\epsilon m}{2}} e^{\tau \phi} f\|+\tau^m \delta^m \|r^{-m} e^{\tau \phi} f\|
\label{hiCarl}
\end{align}
for any $f\in C^\infty_0(\mathbb B_{R_0}(x_0)\backslash \mathbb B_{\delta}(x_0))$ and  $\tau>C(1+\lambda^{\frac{1}{2m}})$. Following the arguments in Proposition 2, we make use of growth estimates (\ref{higher}) and Carleman estimates (\ref{hiCarl}) to establish the doubling inequalities
\begin{equation}
\|e_\lambda\|_{L^\infty(\mathbb B_{2r}(x_0))}\leq e^{C \lambda^\frac{1}{2m}}\|e_\lambda\|_{L^\infty(\mathbb B_{r}(x_0))}
\label{highdouble}
\end{equation}
for any $x_0\in \overline{\Omega}$ and $0<r\leq \frac{d}{4}$. As in the proof of Theorem 3, the doubling inequality (\ref{highdouble}) and complex growth lemma  imply the measure of nodal sets of eigenfunctions in (\ref{moder}). Thus, we can obtain the desired estimates
\begin{equation} H^{n-1}\{x\in
\Omega|e_\lambda(x)=0\}\leq
C\lambda^\frac{1}{2m}.
\label{achiv}
\end{equation}

The same approach can also be applied for the eigenvalue problem (\ref{Navier}) with the exactly same upper bounds of nodal sets in (\ref{achiv}).
Thus, the corollary is completed.

\end{proof}

\end{document}